\documentclass[10pt]{amsart}
\usepackage{amsmath,amscd,amssymb,amsbsy,epsfig}

\newcommand{\ZZ}{\mathbb{Z}}
\newcommand{\QQ}{\mathbb{Q}}

\newcommand{\CC}{\mathbb{C}}
\newcommand{\FF}{\mathbb{F}}

\newcommand{\Supp}{{\rm Supp}}
\newcommand{\sign}{{\rm sign}}

\newcommand{\St}{{\rm Fix}}
\newcommand{\Irr}{{\rm Irr}}
\newcommand{\bv}{{\bar{v}}}

\newcommand{\bc}{{\boldsymbol{c}}}
\newcommand{\blambda}{{\boldsymbol{\lambda}}}
\newcommand{\brh}{{\boldsymbol{rh}}}
\newcommand{\RHT}{RHT}
\newcommand{\bRHT}{{\boldsymbol{\RHT}}}
\newcommand{\bT}{{\boldsymbol{T}}}
\newcommand{\bP}{{\boldsymbol{P}}}
\newcommand{\bQ}{{\boldsymbol{Q}}}

\newcommand{\Inv}{\operatorname{Inv}}
\newcommand{\Pair}{{\operatorname{Pair}}}

\newcommand{\hgt}{\operatorname{ht}}

\newcommand{\Des}{{\rm{Des}}}

\newtheorem{theorem}{Theorem}[section]
\newtheorem{corollary}[theorem]{Corollary}
\newtheorem{proposition}[theorem]{Proposition}
\newtheorem{conjecture}[theorem]{Conjecture}
\newtheorem{lemma}[theorem]{Lemma}
\newtheorem{defn}[theorem]{Definition}

\newtheorem{remark}[theorem]{Remark}
\newtheorem{question}[theorem]{Question}

\newtheorem{observation}[theorem]{Observation}

\numberwithin{figure}{section}

\begin{document}
\title{A Gelfand Model for Wreath Products}
\bibliographystyle{acm}
\author{Ron M. Adin}
\address{Department of Mathematics, Bar-Ilan University,
Ramat-Gan 52900, Israel}
\email{radin@math.biu.ac.il}
\author{Alexander Postnikov}
\address{Department of Applied Mathematics, Massachusetts Institute of Technology,
Cambridge, MA 02139, USA}
\email{apost@math.mit.edu}
\author{ Yuval Roichman}
\address{Department of Mathematics, Bar-Ilan University,
Ramat-Gan 52900, Israel}
\email{yuvalr@math.biu.ac.il}

\date{February 11, 2008}

\maketitle

\begin{abstract}
A Gelafand model for wreath products $\ZZ_r\wr S_n$ is
constructed. The proof relies on a combinatorial interpretation of
the characters of the model, extending a classical result of
Frobenius and Schur.
\end{abstract}

\section{Introduction}\label{s.intro-section}

A complex representation of a group $G$ is called a {\em Gelfand
model} for $G$, or simply a {\em model},  if it is equivalent to
the multiplicity-free direct sum of all the irreducible
representations of $G$. The problem of constructing models was
introduced by Bernstein, Gelfand and Gelfand~\cite{BGG}.
Constructions of models for the symmetric group, using
representations induced from centralizers, were found by
Klyachko~\cite{K1, K2} and by Inglis, Richardson and
Saxl~\cite{Saxl}; see also~\cite{B, R, A, AA, AB}.
In this paper we determine an explicit and simple combinatorial
action which gives a model for wreath products $\ZZ_r\wr S_n$, and
in particular for the Weyl groups of type $B$. For $r=1$ (i.e.,
for the symmetric group) the construction is identical with the
one given
in~\cite{Verma, APR}. 
The proof relies on a combinatorial interpretation of the
characters, extending a classical result of Frobenius and Schur.

\medskip
If all the (irreducible) representations of a finite group are
real then, by a result of Frobenius and Schur, the character-value
of a model at a group element is the number of square roots of
this element in the group. We are concerned in this paper with
$G(r,n) = \ZZ_r \wr S_n$, the wreath product of a cyclic group
$\ZZ_r$ with a symmetric group $S_n$. For $r > 2$ this group is
not real, and Frobenius' theorem does not apply. It will be shown
that the character-value of a model at an element of $G(r,n)$ is
the number of ``absolute square roots'' of this element in the
group; see Theorem~\ref{t.main} below.

\medskip

The rest of the paper is organized as follows.
The construction of the model is described in Subsection~\ref{s.main-results}.
Necessary preliminaries and notation are given in Section~\ref{s.preliminaries}.
The combinatorial interpretation of the characters of the model is
described in Section~\ref{s.formula} (Theorem~\ref{t.main}). Two
proofs for this interpretation are given. A direct combinatorial
proof, using the Murnaghan-Nakayama rule, is given in
Section~\ref{s.direct-proof}. The second proof combines the
properties of the generalized Robinson-Schensted algorithm for
wreath products, due to Stanton and White, with a generalized
Frobenius-Schur formula due to Bump and Ginzburg; see
Section~\ref{s.second-proof}. The main theorem
(Theorem~\ref{t.model}) is proved in
Section~\ref{s.proof-of-model}. The proof applies generalized
Frobenius-Schur character formula (Theorem~\ref{t.main}) together
with Corollary~\ref{t.root}. Section~\ref{s.final-section} ends
the paper with final remarks and open problems.


\subsection{Main Result}\label{s.main-results}


\begin{defn}
Consider the natural representation $\varphi: \ZZ_r\wr S_n \to
GL_n(\CC)$. An element $\pi\in \ZZ_r\wr S_n$ is called {\em
symmetric} if $\varphi(\pi)$ is a (complex) symmetric matrix, that is
$\varphi(\pi)^t = \varphi(\pi)$.
\end{defn}

Denote by $I_{r,n}$ - the set of symmetric elements in $\ZZ_r\wr
S_n$, and let $V_{r,n}:={\rm{span}}_\QQ\{C_v\ |\ v\in I_{r,n}\}$
be a vector space over $\QQ$ with a basis indexed by the symmetric
elements.

Recall that that each element $v\in \ZZ_r\wr S_n$ may be
represented by a pair $(\sigma,z)$ where $\sigma\in S_n$ and $z\in
\ZZ_r^n$. Denote $|v|:=\sigma$ and $\omega:=e^{2\pi i/ r}$. Let
$S$ be the standard generating set of simple complex reflections
in $\ZZ_r\wr S_n$; namely, $S=\{s_0,s_1,\dots,s_{n-1}\}$ where
$s_0=([1,2,\dots,n],(1,0,\dots,0))$ and
$s_i=([1,2,\dots,i-1,i+1,i,i+2,\dots,n],(0,\dots,0))$ ($i>0$).
Note that for $r>2$ $s_0$ is not an involution.

\medskip

Define a map $\rho: S \to GL(V_{r,n})$ by
\begin{equation}\label{e.linear_action}
\rho(s_i)C_v := {\rm{sign}}(i;\ v)\cdot C_{s_ivs_i}\qquad (0\le i\le n-1, v\in I_{r,n}),
\end{equation}
where
$$ 
{\rm{sign}}(i;\ v):=
\begin{cases}
   -1,& \text{ if } s_i v s_i=v \text{ and } s_i\in \Des(|v|);\\
   1, & \text{ otherwise }
\end{cases}
\qquad(\forall i > 0)
$$ 
(namely, ${\rm sign}(i,\ v) = -1$ iff $|v|\in S_n$ permutes $i$ and $i+1$),
while
$$ 
{\rm{sign}}(0;\ v):=
\begin{cases}
   -1,& \text{ if }  v(1)= 1\cdot \omega^{-1} \text{ and } r \text{ is even};\\
   1, & \text{ otherwise,}
\end{cases}
$$ 
In particular, if $r$ is odd then $\sign(0; v)$ is always $1$.

\medskip

Using a generalized Frobenius-Schur character formula
(Theorem~\ref{t.main}) together with Corollary~\ref{t.root} we
prove

\begin{theorem}\label{t.model}
$\rho$ extends to a Gelfand model for $\ZZ_r\wr S_n$.
\end{theorem}

\begin{remark}\ \rm{
Note that for $r=2$ (i.e., for the Weyl group of type $B$)
$$
{\rm{sign}}(0;\ v):=
\begin{cases}
   -1,   & \text{ if } s_0vs_0=v \text { and } s_0\in \Des(v);\\
   1, & \text{ otherwise. }
\end{cases}
$$
Thus Theorem~\ref{t.model} implies~\cite[Theorem 5.1.1]{APR},
which was stated there without proof.}
\end{remark}

\section{Preliminaries and Notation}\label{s.preliminaries}

Let $S_n$ be the symmetric group on $n$ letters, $\ZZ_r$ the
cyclic group of order $r$ (realized as, the additive group
of integers modulo $r$), and $G(r,n) = \ZZ_r \wr S_n$ their wreath
product:
$$
G(r,n) := \{g = (\sigma, (c_1,\ldots,c_n))\,|\,\sigma\in S_n,\,c_i\in\ZZ_r\,(\forall i)\}
$$
with the group operation
$$
(\sigma, (c_1,\ldots,c_n)) \cdot (\tau, (d_1,\ldots,d_n)) :=
(\sigma\tau, (c_{\tau^{-1}(1)} + d_1, \ldots, c_{\tau^{-1}(n)} + d_n)).
$$

\bigskip

The Murnaghan-Nakayama rule is an explicit formula for the character values
of irreducible representations of $S_n$ (and of $G(r,n)$). We shall first
describe the formula for $S_n$ (for later use in our proofs), and
then give its generalization to $G(r,n)$.

A {\em rim hook tableau} of shape $\lambda$ is a sequence
$$
\emptyset = \lambda^{(0)} \subset \ldots \subset \lambda^{(t)} = \lambda
$$
of Young diagrams such that each consecutive difference
$rh_i := \lambda^{(i)} \setminus \lambda^{(i-1)}$ ($1\le i\le t$)
is a non-empty {\em rim hook} (or {\em border strip}), namely
a connected skew diagram ``of width 1''
(i.e., containing no $2\times 2$ square).
The sequence can be described by one tableau $T$ of shape $\lambda$
in which the cells of each rim hook $rh_i$ are marked $i$.
The {\em length} $l(rh_i)$ of a rim hook $rh_i$ is
the number of cells it contains;
its {\em height} $\hgt(rh_i)$ is the height difference between
its two extreme cells.

\begin{proposition}\label{t.murnaghan_sn}{\rm (Murnaghan-Nakayama rule for $S_n$)}
Fix an ordering $\bc = (c_1,\ldots,c_m)$ of the disjoint cycles of
a permutation $\sigma\in S_n$, and let $l(c_i)$ be the length of $c_i$.
For any partition $\lambda$ of $n$ let $\chi^\lambda$ be the corresponding
irreducible character of $S_n$. Then
$$
\chi^\lambda(\sigma) = \sum_{T\in \RHT_\bc(\lambda)} \prod_{i=1}^{m} (-1)^{\hgt(rh_i)},
$$
where $\RHT_\bc(\lambda)$ is the set of all rim hook tableaux of shape $\lambda$ with
$l(rh_i) = l(c_i)$ ($\forall i$).
\end{proposition}

In order to describe the characters of $G(r,n)$ let us recall the notions of
$r$-partite partitions and tableaux.
An {\em $r$-partite partition of $n$} is an $r$-tuple
$\blambda = (\lambda_0,\dots,\lambda_{r-1})$ such that
each $\lambda_i$ is a partition of a nonnegative integer $n_i$
and $n_0+\ldots+n_{r-1}=n$.
(We shall usually use \textbf{boldface} to denote $r$-partite concepts.)
An {\em $r$-partite standard Young tableau} of shape $\blambda$ is obtained by
inserting the integers $1,\ldots,n$ bijectively into the cells of the
corresponding diagrams such that entries increase along
each row and column of each diagram.
An {\em $r$-partite rim hook tableau} of shape $\blambda$ is a sequence
$$
\emptyset = \blambda^{(0)} \subset \ldots \subset \blambda^{(t)} = \blambda
$$
of $r$-partite partitions (diagrams) such that each consecutive difference
$\brh_i := \blambda^{(i)} \setminus \blambda^{(i-1)}$ ($1\le i\le t$),
as an $r$-tuple of skew shapes, has $r-1$ empty parts and one non-empty part
which is a rim hook $rh_i$:
$\brh_i = (\ldots,\emptyset,rh_i,\emptyset,\ldots)$.
Let $f(i)\in[0,r-1]$ be the index of the non-empty part of $\brh_i$.
Again, an $r$-partite rim hook tableau can be described by
an $r$-partite tableau in which the cells of each rim hook $rh_i$
are marked $i$ ($1\le i\le t$).

The conjugacy classes of $G(r,n)$ are described by the cycle structure
of the underlying permutations in $S_n$, sub-classified by
the sum of colors (in $\ZZ_r$) in each cycle.
These correspond to $r$-partite partitions.
The irreducible representations of $G(r,n)$ may be indexed by the same
combinatorial objects.
A construction of the irreducible representation $S^\blambda$ indexed by
each $r$-partite partition $\blambda$ was given by Specht in the thirties.

The dimension of $S^\blambda$ is equal to the number of $r$-partite
standard Young tableaux of shape $\blambda$.
It follows that the number of pairs of $r$-partite standard Young tableaux
of the same shape is equal to the cardinality of $G(r,n)$.
A bijective proof was given by Stanton and White~\cite{SW},
using a generalized Robinson-Schensted algorithm;
see Section~\ref{s.second-proof}.

\begin{proposition}\label{t.murnaghan_grn}{\rm (Murnaghan-Nakayama rule for $\ZZ_r\wr S_n$)}
Fix an ordering $\bc = (c_1,\ldots,c_m)$ of the disjoint cycles of
a colored permutation $g\in \ZZ_r\wr S_n$. Let $l(c_i)$ be the
length of $c_i$ and let $z(c_i)\in \ZZ_r$ be its color (the sum of
colors of its elements). For any $r$-partite partition $\blambda$
of $n$, let $\chi^\blambda$ be the corresponding irreducible
character of $\ZZ_r\wr S_n$. Then
$$
\chi^\blambda(g) = \sum_{\bT\in \bRHT_\bc(\blambda)}
                   \prod_{i=1}^{m} (-1)^{\hgt(rh_i)} \omega^{f(i) \cdot z(c_i)},
$$
where $\bRHT_\bc(\blambda)$ is the set of all $r$-partite rim hook tableaux
of shape $\blambda$ such that $l(rh_i) = l(c_i)$ ($\forall i$);
$f(i)\in [0,r-1]$ is the index of the nonempty part $rh_i$ of $\brh_i$, as above;
and $\omega := e^{2\pi i/r}$.
\end{proposition}

\section{A Character Formula}\label{s.formula}

Let $\Irr(G)$ be the set of irreducible complex characters of a
finite group $G$. A classical result of Frobenius and Schur~\cite{FS}
(see, e.g.,~\cite{Is}) is
\begin{proposition}{\rm (Frobenius-Schur)}\label{t.frobenius}
For any finite group $G$ and any $g\in G$,
$$
\#\{v\in G\,|\,v^2 = g\} = \sum_{\chi\in\Irr(G)} \epsilon(\chi)\chi(g),
$$
where
$$
\epsilon(\chi) := \begin{cases}
1,&\hbox{if $\chi$ is afforded by a real representation;}\\
-1,&\hbox{if $\chi$ is real-valued, but is not afforded by a real representation;}\\
0,&\hbox{if $\chi$ is not real-valued.}
\end{cases}
$$
\end{proposition}
In particular, if every character of $G$ is afforded by a real representation then
$$
\sum_{\chi\in\Irr(G)} \chi(g) = \#\{v\in G\,|\,v^2 = g\}\qquad(\forall g\in G).
$$
For $r > 2$ the group $G(r,n)$ has non-real representations, so that a Gelfand
model for it does not give the number of square roots of an
element. What does it give?

\begin{defn}\label{d.bar}
For
$$
v = (\sigma, (z_1,\ldots,z_n))\in G(r,n)\qquad(\sigma\in
S_n,\,z_i\in\ZZ_r\,(\forall i)),
$$
define the {\em bar operation}
$$
\bv := (\sigma, (-z_1,\ldots,-z_n)).
$$
An element $v\in G(r,n)$ is an {\em absolute square root} of $g\in G(r,n)$
if $v\cdot \bv=g$. An element $v\in G(r,n)$ is an {\em absolute involution}
if $v\cdot \bv=id$.
\end{defn}

\begin{remark}
Elements of $G(r,n)$ may also be represented by monomial matrices:
$v = (\sigma,(z_1,\ldots,z_n))$ corresponds to $M = (m_{ij})$, where
$$
m_{ij} = \begin{cases}
\omega^{z_j},& \text{ if } i=\sigma(j);\\
0,& \text{ otherwise.}
\end{cases}
$$
Then $\bv$ corresponds to $\bar{M} = (\bar{m}_{ij})$, the (entry-wise)
{\em complex conjugate} of $M$.
\end{remark}

\begin{theorem}\label{t.main}
For any $g\in G(r,n) = \ZZ_r \wr S_n$,
$$
\sum_{\chi\in\Irr(G)} \chi(g) = \#\{v\in G\,|\,v \cdot \bv = g\}.
$$
\end{theorem}
In particular, these sums are nonnegative integers.

Two proofs of Theorem~\ref{t.main} will be given in the next two
sections.

\section{A Combinatorial Proof of Theorem~\ref{t.main}}\label{s.direct-proof}

A direct proof, using the Murnaghan-Nakayama rule, is given in
this section.

\begin{lemma}\label{t.chi}
Let $g\in G = G(r,n)$, and fix an ordering $\bc = (c_1,\ldots,c_m)$ of
the disjoint cycles of $g$. Then
$$
\sum_{\chi\in\Irr(G)} \chi(g)
= \sum_f \omega^{\alpha(f)} \prod_{j=0}^{r-1} \sum_{\lambda_j\vdash n_j}
  \chi^{\lambda_j}(\sigma_j),
$$
where the sum is over all functions $f: [m] \to [0,r-1]$,
$z(c_i)$ and $l(c_i)$ are as in Proposition~\ref{t.murnaghan_grn},
$$
\alpha(f) := \sum_{i=1}^{m} f(i) \cdot z(c_i) \in \ZZ_r,
$$
$$
n_j := \sum_{i\in f^{-1}(j)} l(c_i)\qquad(0\le j\le r-1),
$$
and $\sigma_j\in S_{n_j}$ is the product of all disjoint cycles $|c_i|$
with $f(i)=j$ $(0\le j\le r-1)$.
\end{lemma}

\begin{proof}
By Proposition~\ref{t.murnaghan_grn} (the Murnaghan-Nakayama rule for $G(r,n)$),
$$
\sum_{\chi\in\Irr(G)} \chi(g) = \sum_{\blambda} \chi^\blambda(g)
= \sum_{\blambda} \sum_{\bT\in \bRHT_\bc(\blambda)}
                  \prod_{i=1}^{m} (-1)^{\hgt(rh_i)} \omega^{f(i) \cdot z(c_i)}.
$$
Recall that each $r$-partite rim hook tableau $\bT$ determines a function
$f:[m]\to [0,r-1]$, where $f(i)$ is the index of the tableau to which a rim hook
is added in step $i$, corresponding to cycle $c_i$ ($1\le i\le m$).
We shall change the order of summation by first summing over the possible functions $f$.
The function $f$ determines which cycles $c_i$ ``go'' to each tableau $T_j$
($0\le j\le r-1$), and therefore also the size
$$
n_j := \sum_{i\in f^{-1}(j)} l(c_i)
$$
of this tableau (but not its shape $\lambda_j$). Also, the product
$$
\omega^{\alpha(f)} := \prod_{i=1}^{m} \omega^{f(i) \cdot z(c_i)}
$$
depends only on the function $f$. Thus
$$
\sum_{\chi\in\Irr(G)} \chi(g)
= \sum_f \omega^{\alpha(f)} \prod_{j=0}^{r-1} \sum_{\lambda_j\vdash n_j}
  \sum_{T_j\in \RHT_{|\bc(j)|}(\lambda_j)} \prod_{i\in f^{-1}(j)} (-1)^{\hgt(rh_i)},
$$
where $|\bc(j)|$ is the sequence of cycles $|c_i|$ with $f(i) = j$, ordered by increasing $i$.

The expression after $\omega^{\alpha(f)}$ is clearly color-free, and depends only on
$\sigma := |g|\in S_n$. Given $\sigma$ and $f$, define
$$
\sigma_j := \text{ product of all disjoint cycles } |c_i| \text{ with } f(i)=j
            \qquad(1\le j\le r-1).
$$
Then $\sigma_j$ permutes a set of size $n_j$, and by abuse of language
we can write $\sigma_j\in S_{n_j}$. By Proposition~\ref{t.murnaghan_sn}
(the Murnaghan-Nakayama rule for $S_n$) we conclude
$$
\sum_{\chi\in\Irr(G)} \chi(g)
= \sum_f \omega^{\alpha(f)} \prod_{j=0}^{r-1} \sum_{\lambda_j\vdash n_j}
  \chi^{\lambda_j}(\sigma_j),
$$
as required.
\end{proof}

\begin{observation}
Let $v\in G = G(r,d)$ be a single colored cycle, and let $w := v \cdot \bar{v}$.
\begin{enumerate}
\item
If $d$ is odd then $w$ is a single colored cycle (of length $d$) with $z(w) = 0$.
Any such $w$ is obtained from $r$ possible cycles $v$.
\item
If $d$ is even then $w$ is a product of two disjoint colored cycles $w_1$ and $w_2$,
each of length $d/2$, and $z(w_1) + z(w_2) = 0$.
Any such $w$ is obtained from $rd/2$ possible cycles $v$.
\end{enumerate}
\end{observation}

\begin{corollary}\label{t.root}
Let $g\in G = G(r,n)$, and fix an ordering $\bc = (c_1,\ldots,c_m)$ of
the disjoint cycles of $g$. Then
$$
\#\{v\in G\,|\,v \cdot \bv = g\} = \prod_{d=1}^{\infty} N_d(\bc),
$$
where
$$
N_d(\bc) :=
\begin{cases}
\sum_{P\in\Pi_d^{2,1}(\bc)} (dr)^{n_2(P)} r^{n_1(P)},&\hbox{if $d$ is odd};\\
\sum_{P\in\Pi_d^{2}(\bc)} (dr)^{n_2(P)},&\hbox{if $d$ is even}.
\end{cases}
$$
Here $\Pi_d^{2,1}(\bc)$ is the set of all partitions $P$ of the set $\{i\,|\,l(c_i) = d\}$
into $n_2(P)$ pairs and $n_1(P)$ singletons such that
$z(c_i) + z(c_j) = 0$ for each pair $\{i,j\}$ in $P$ and
$z(c_i) = 0$ for each singleton $\{i\}$ in $P$.
$\Pi_d^{2}(\bc)$ is defined similarly, where only pairs are allowed.
\end{corollary}

\begin{corollary}\label{t.req1}$(r=1)$
Let $\sigma\in S_n$, and fix an ordering $\bc = (c_1,\ldots,c_m)$ of
the disjoint cycles of $\sigma$. Then
$$
\#\{v\in S_n\,|\,v^2 = \sigma\} = \prod_{d=1}^{\infty} N_d(\bc),
$$
where
$$
N_d(\bc) :=
\begin{cases}
\sum_{P\in\Pi_d^{2,1}(\bc)} d^{n_2(P)},&\hbox{if $d$ is odd};\\
\sum_{P\in\Pi_d^{2}(\bc)} d^{n_2(P)},&\hbox{if $d$ is even}.
\end{cases}
$$
Here $\Pi_d^{2,1}(\bc)$ is the set of all partitions $P$ of the set $\{i\,|\,l(c_i) = d\}$
into $n_2(P)$ pairs and $n_1(P)$ singletons, and
$\Pi_d^{2}(\bc)$ is the set of all partitions of this set into $n_2(P)$ pairs
(with no singletons).
\end{corollary}

\begin{proof} (of Theorem~\ref{t.main})
For $0\le j\le r-1$ let $\bc(j)$ be the sequence of cycles $c_i$ with $f(i) = j$,
ordered by increasing $i$.
By Lemma~\ref{t.chi}, Proposition~\ref{t.frobenius} for the symmetric groups
$G = S_{n_j}$ (all of whose representations are real),
and Corollary~\ref{t.req1} for these groups:
\begin{eqnarray*}
\sum_{\chi\in\Irr(G)} \chi(g)
&=& \sum_f \omega^{\alpha(f)} \prod_{j=0}^{r-1} \sum_{\lambda_j\vdash n_j}
    \chi^{\lambda_j}(\sigma_j)\\
&=& \sum_f \omega^{\alpha(f)} \prod_{j=0}^{r-1} \#\{v_j\in S_{n_j}\,|\,v_j^2 = \sigma_j\}\\
&=& \sum_f \omega^{\alpha(f)} \prod_{j=0}^{r-1}
    \prod_{d=1}^{\infty} \sum_{P_{j,d}\in\Pi_d^{2(,1)}(|\bc(j)|)} d^{n_2(P_{j,d})}.
\end{eqnarray*}
Here we used the notations of the previous lemmas and corollaries, as well as the the
shorthand notation
$$
\Pi_d^{2(,1)}(|\bc(j)|) := \begin{cases}
\Pi_d^{2,1}(|\bc(j)|),& \text{ if $d$ is odd;}\\
\Pi_d^{2}(|\bc(j)|),& \text{ if $d$ is even.}
\end{cases}
$$
A more succinct expression is
$$
\sum_{\chi\in\Irr(G)} \chi(g)
= \sum_f \omega^{\alpha(f)} \sum_{P\in\Pi_f^{2(,1)}(|\bc|)} \beta(P),
$$
where $\Pi_f^{2(,1)}(|\bc|)$ is the set of all partitions $P$ of the set $[m]$
into pairs and singletons such that,
for each pair $\{i,j\}$ in $P$, $l(c_i) = l(c_j)$ and $f(i) = f(j)$,
and for each singleton $\{i\}$ in $P$ the length $l(c_i)$ is odd;
and for any such partition $P$
$$
\beta(P) := \prod_{d=1}^{\infty} d^{n_{2,d}(P)},
$$
where $n_{2,d}(P)$ is the number of pairs $\{i,j\}$ in $P$ such that
$l(c_i) = l(c_j) = d$.

The next step is to change the order of summation:
$$
\sum_{\chi\in\Irr(G)} \chi(g)
= \sum_{P\in\Pi^{2(,1)}(|\bc|)} \beta(P) \sum_{f\in F_P} \omega^{\alpha(f)},
$$
where $\Pi^{2(,1)}(|\bc|)$ is the set of all partitions as above but {\em without}
the restriction $f(i) = f(j)$; and where $F_P$ is the set of all functions
$f: [m] \to [0,r-1]$ such that $f(i) = f(j)$ whenever $\{i,j\}$ is a pair in $P$.
This requirement means that $f$ is constant on each part of $P$, and therefore
determines a unique function $f': P \to [0,r-1]$, where $P$ is viewed as a set
of pairs and singletons.
For each part $p\in P$ let $l(p)$ be its length
($= l(c_i) = l(c_j)$ if $p = \{i,j\}$, $= l(c_i)$ if $p = \{i\}$)
and $z(p)$ its color
($= z(c_i) + z(c_j)$ if $p = \{i,j\}$, $= z(c_i)$ if $p = \{i\})$.
Recalling the definition of $\alpha(f)$ from Lemma~\ref{t.chi},
\begin{eqnarray*}
\sum_{\chi\in\Irr(G)} \chi(g)
&=& \sum_{P\in\Pi^{2(,1)}(|\bc|)} \beta(P)
    \sum_{f':P\to [0,r-1]} \prod_{p\in P} \omega^{f'(p) \cdot z(p)}\\
&=& \sum_{P\in\Pi^{2(,1)}(|\bc|)} \beta(P)
    \prod_{p\in P} \sum_{j=0}^{r-1} \omega^{j \cdot z(p)}.
\end{eqnarray*}
Now use the observation
$$
\sum_{j=0}^{r-1} \omega^{j \cdot a} = \begin{cases}
r,&\hbox{if $0 =a\in \ZZ_r$;}\\
0,&\hbox{if $0\ne a\in \ZZ_r$}
\end{cases}
$$
to simplify:
$$
\sum_{\chi\in\Irr(G)} \chi(g)
= \sum_{P\in\Pi^{2(,1)}(|\bc|) \atop z(p)=0\,(\forall p\in P)} \beta(P) \cdot r^{n_2(P) + n_1(P)}.
$$
Denote by $\Pi^{2(,1)}(\bc)$ the set of partitions $P\in\Pi^{2(,1)}(|\bc|)$ such that
$z(p)=0$ for all $p\in P$, in accordance with the notation in Corollary~\ref{t.root}.
Recalling the definition of $\beta(P)$,
\begin{eqnarray*}
\sum_{\chi\in\Irr(G)} \chi(g)
&=& \sum_{P\in\Pi^{2(,1)}(\bc)} \beta(P) \cdot r^{n_2(P) + n_1(P)}\\
&=& \prod_{d=1}^{\infty} \sum_{P_d\in\Pi_d^{2(,1)}(\bc)} (dr)^{n_2(P_d)} r^{n_1(P_d)}.
\end{eqnarray*}
Comparison to Corollary~\ref{t.root} completes the proof:
$$
\sum_{\chi\in\Irr(G)} \chi(g) = \#\{v\in G\,|\,v \cdot \bv = g\}.
$$
\end{proof}

\section{A Second Proof of Theorem~\ref{t.main}}\label{s.second-proof}



For a complex matrix $A$ let $\bar A$ be the matrix obtained from $A$
by complex conjugation of each entry, and let $A^t$ be the transposed
matrix. Consider the $n$-dimensional natural representation $\varphi$
of $G(r,n)$. An element $\pi\in G(r,n)$ is called {\em symmetric} if
$\varphi(\pi)$ is a symmetric matrix, that is $\varphi(\pi)^t=\varphi(\pi)$.

\begin{observation}\label{t.matrices}
Let $\pi = (\sigma, (z_1,\ldots,z_n))\in G(r,n)$. Then:
\begin{enumerate}
\item
$\overline{\varphi (\pi)} = \varphi (\bar \pi)$, where $\bar\pi:=(\sigma, (-z_1,\ldots,-z_n)$.
\item
$ \varphi (\pi)^t= \varphi(\pi^t)$, where $\pi^t:=(\bar\pi)^{-1}$.
\item
$\pi\in G(r,n)$ is symmetric if and only if it is an absolute involution:
$\pi \cdot \bar\pi = id$.
\end{enumerate}
\end{observation}


Recall the notions of $r$-partite partitions and tableaux from
Section~\ref{s.preliminaries}. Stanton and White~\cite{SW} described
and studied the following generalization of the Robinson-Schensted
algorithm to $G(r,n)$. Given $\pi\in G(r,n)$ produce a pair
$(\bP,\bQ)$ of $r$-partite standard Young tableaux,
where $\bP = (P_0,\ldots,P_{r-1})$ and $\bQ = (Q_0,\ldots,Q_{r-1})$,
by mapping the letters colored by $i$ to the $i$-th tableaux $P_i$
according to the usual Robinson-Schensted algorithm;
their positions are recorded in the tableaux $Q_i$. This gives a
bijection from the set of all elements in $G(r,n)$ to the set of
all pairs of $r$-partite standard Young tableaux of same shape.

The following lemma is a reformulation of~\cite[Corollary 29]{SW}.

\begin{lemma}\label{t.RS2}
For every $\pi\in G(r,n)$
$$
\pi {\stackrel{\rm RS}{\longrightarrow}} (\bP,\bQ)
\Longleftrightarrow
\pi^t {\stackrel{\rm RS}{\longrightarrow}} (\bQ,\bP).
$$
\end{lemma}

We deduce

\begin{corollary}\label{t.dim}
The dimension of the model of $G(r,n)$ is equal to the number of
absolute involutions in $G(r,n)$.
\end{corollary}

\begin{proof}
It is well known that the dimension of the irreducible $G(r,n)$
representation indexed by an $r$-partite partition $\blambda$ is
the number of $r$-partite SYT of shape $\blambda$.
By Lemma~\ref{t.RS2}, the number of symmetric elements $\pi\in G(r,n)$
is equal to the number of $r$-partite SYT.
Observation~\ref{t.matrices}(3) completes the proof.
\end{proof}

It should be noted that corollary~\ref{t.dim} is analogous to a
remarkable theorem of Klyachko~\cite{K1, K2} and Gow~\cite{G}
regarding the dimension of the model for $GL_n(\FF_q)$.

\bigskip

The following generalization of the Frobenius-Schur theorem (Proposition~\ref{t.frobenius})
was proved by Bump and Ginzburg.

\begin{proposition}\label{t.bg1}~\cite[Theorem 3]{BG}
Let $G$ be a finite group, let $\tau$ be an automorphism of $G$
satisfying $\tau^2=1$, and let $z\in G$ such that $z^2=id$.
If
$$
\sum_{\rho\in \Irr(g)} \chi(id)=\#\{w\in G:\ w \cdot \tau(w)=z\}
$$
then
$$
\sum_{\rho\in \Irr(g)} \chi(g)=\#\{w\in G:\ w \cdot \tau(w)=gz\}\qquad (\forall g\in G).
$$
\end{proposition}

Now let $G := G(r,n)$, $\tau$ be the bar operation from Definition~\ref{d.bar},
and $z:=id$.
By Corollary~\ref{t.dim}, the assumptions of Theorem~\ref{t.bg1} are satisfied,
implying Theorem~\ref{t.main}.

\qed

\section{Proof of Theorem~\ref{t.model}}\label{s.proof-of-model}

We will first prove the theorem for odd $r$. The necessary
modifications for the more complicated case of even $r$ will be
indicated afterwards.

\subsection{The Case of Odd $r$}

\subsubsection{Part 1}
We start by showing that $\rho$ can be extended to a group homomorphism.
Recall the definition of the {\em inversion set} of a
permutation $\sigma\in S_n$,
$$
\Inv(\sigma):=\{\,\{i,j\}:\,(j - i)\cdot(\sigma(j) - \sigma(i)) < 0\}.
$$
For each involution $v\in S_n$ (including $v=id$), let $\Pair(v)$ be the set of
pairs $\{i,j\}$ such that $(i,j)$ is a 2-cycle of $v$.

\begin{defn}\label{d.sign_o}
For an element $\pi\in G(r,n)$ and an absolute involution $w\in I_{r,n}$ let
$$
\sign_o(\pi,w):=(-1)^{\#(\Inv(|\pi|) \cap \Pair(|w|))}.
$$
\end{defn}

Define a map $\rho: G(r,n)\to GL(V_{r,n})$ by
$$ 
\rho(\pi)C_w := \sign_o(\pi,w) \cdot C_{\pi w\pi^t}
\qquad (\forall \pi\in G(r,n), w\in I_{r,n}).
$$ 

\medskip

One can verify that this definition of $\rho$ coincides,
on the set $S$ of generators of $G(r,n)$,
with the previous definition~(\ref{e.linear_action}).
It thus suffices to show that $\rho$ is a group homomorphism.
By definition of $\rho$, it suffices to prove that
\begin{equation}\label{e.sign-1-hom}
\sign_o(\pi_2\pi_1,w)=\sign_o(\pi_1,w)\cdot \sign_o(\pi_2,\pi_1 w{\pi_1^t}).
\end{equation}
Indeed, let $X[\text{condition}]$ be $-1$ if the condition holds, and $1$ otherwise.
Then, for any $\pi_1, \pi_2\in G(r,n)$, $w\in I_{r,n}$ and $i \ne j$,
denoting $\sigma_1 := |\pi_1|$, $\sigma_2 := |\pi_2|$ and $v := |w|$:
$$
\{i,j\}\in \Pair(v) \iff \{\sigma_1(i),\sigma_1(j)\}\in \Pair(\sigma_1 v \sigma_1^{-1})
$$
and
$$
X[\{i,j\}\in \Inv(\sigma_2\sigma_1)] =
X[\{i,j\}\in \Inv(\sigma_1)] \cdot
X[\{\sigma_1(i),\sigma_1(j)\}\in \Inv(\sigma_2)].
$$
Thus
$$
X[\{i,j\}\in \Inv(|\pi_2\pi_1|) \cap \Pair(|w|)] =
$$
$$
X[\{i,j\}\in \Inv(|\pi_1|) \cap \Pair(|w|)] \cdot
X[\{|\pi_1(i)|,|\pi_1(j)|\}\in \Inv(|\pi_2|) \cap \Pair(|\pi_1 w \pi_1^t|)],
$$
which implies (\ref{e.sign-1-hom}) by taking a product over all possible pairs $\{i,j\}$.

\subsubsection{Part 2.}
For an arbitrary element $\pi\in G(r,n)$ let
$$
\St(\pi):=\{w\in I_{r,n}:\,\pi w\pi^t=w\} = \{w\in I_{r,n}:\ \pi w=w\bar \pi\}.
$$
We shall prove that $\rho$ is a model for $G(r,n)$ by showing that
\begin{equation}\label{e.sign_chi}
\sum\limits_{w\in \St(\pi)} \sign_o(\pi,w) =
\sum\limits_{\chi\in \Irr(G(r,n))} \chi(\pi)\qquad(\forall \pi\in G(r,n)).
\end{equation}
%
Let $\pi=(\sigma,z)\in G(r,n)$.
For each $d\ge 1$ let $\Supp_d(\pi)$ be the set of all $i\in [n]$
that belong to a cycle of length $d$ in $\sigma$,
and let $\sigma_d$, $z_d$ and $\pi_d = (\sigma_d,z_d)$ be
the restrictions to $\Supp_d(\pi)$ of $\sigma$, $z$ and $\pi$.
It is clear that if $w\in \St(\pi)$ and $\{i,j\}\in Pair(|w|)$
then, since $|w|$ and $\sigma$ commute,
$i$ and $j$ belong to cycles of the same length in $\sigma$.
Thus we can write (uniquely) $w = w_1 \cdots w_n$, where
each $w_d$ is supported within $\Supp_d(\pi)$.
It is also clear that
$$
w_d\in \St(\pi_d)\qquad(\forall d)
$$
and
$$
\sign_o(\pi,w) = \prod_{d=1}^{n} \sign_o(\pi_d,w_d),
$$
so that
$$
\sum_{w\in \St(\pi)} \sign_o(\pi,w) =
\prod_{d =1}^{n} \sum_{w_d\in \St(\pi_d)} \sign_o(\pi_d,w_d).
$$
On the other hand, by Theorem~\ref{t.main} and Corollary~\ref{t.root},
\begin{eqnarray*}
\sum_{\chi\in\Irr(G(r,n))} \chi(\pi)
&=& \#\{g\in G(r,n)\,|\,g \cdot \bar{g} = \pi\} \\
&=& \prod_{d=1}^{n} \#\{g_d\in G(r,n_d)\,|\,g_d \cdot \overline{g_d} = \pi_d\},
\end{eqnarray*}
where $n_d$ is the size of $\Supp_d(\pi)$.

These observations about the multiplicative property of both sides of~(\ref{e.sign_chi})
make it sufficient to prove~(\ref{e.sign_chi}) for $\pi$ with all cycles of the same length.
Indeed, let $\pi\in G(r,md)$ have $m$ cycles, each of length $d$,
ordered arbitrarily $\bc = (c_1,\ldots,c_m)$.
Again, by Theorem~\ref{t.main} and Corollary~\ref{t.root},
\begin{eqnarray*}
\sum_{\chi\in\Irr(G)} \chi(\pi)
&=& \sum_{P\in\Pi^{2(,1)}(\bc)} (dr)^{n_2(P)} r^{n_1(P)}\\
&=& \begin{cases}
\sum_{P\in\Pi^{2,1}(\bc)} (dr)^{n_2(P)} r^{n_1(P)},&\hbox{if $d$ is odd};\\
\sum_{P\in\Pi^{2}(\bc)} (dr)^{n_2(P)},&\hbox{if $d$ is even}.
\end{cases}
\end{eqnarray*}
Here $\Pi^{2(,1)}(\bc)$ is the set of all partitions $P$ of the set
$[m]$ into $n_2(P)$ pairs and $n_1(P)$ singletons such that
$z(c_i) + z(c_j) = 0$ for each pair $\{i,j\}$ in $P$ and $z(c_i) =
0$ for each singleton $\{i\}$ in $P$, and where we require $n_1(P)
= 0$ if $d$ is even.

It remains to show that for every $\pi\in G(r,md)$ of cycle type
$d^m$ as above,
\begin{eqnarray}\label{e.eq52}
\sum\limits_{w\in \St(\pi)} \sign_o(\pi,w) &=&
\sum_{P\in\Pi^{2(,1)}(\bc)} (dr)^{n_2(P)} r^{n_1(P)}.
\end{eqnarray}
Let $w\in \St(\pi)$. Then $|w|\in S_{md}$ is an involution. Choosing
$i_0\in [md]$ there are three cases to analyze:

\medskip

\noindent{\bf Case (1):} $|w(i_0)|=i_0$.

\noindent
Then there exists a cycle $c=(i_0,\dots,i_{d-1})$ of $|\pi|$
such that $|w(i_t)|=i_t$ for all $0\le t\le d-1$ and
$$ 
\pi w = w \bar{\pi} \iff 
z_w(i_t) + z_\pi(i_t) = -z_\pi(i_t) + z_w(i_{t+1}) \qquad(0\le t\le d-1),
$$ 
where $t+1$ is computed mod $d$.
Summing over $0\le t\le d-1$ gives
$$
2z_\pi(c) = 2\sum\limits_{t=0}^{d-1} z_\pi(i_t)=0
$$
(in $\ZZ_r$). Since $r$ is odd, this implies that
$$
z_\pi(c)=0.
$$
Given $\pi$, a choice of $z_w(i_0)\in \ZZ_r$ determines uniquely
$z_w(i_t)$ for all $0\le t\le d-1$.

\medskip

\noindent{\bf Case (2):}
$i_0$ and $|w(i_0)|$ are distinct and belong to different cycles of $|\pi|$.

\noindent
Then there are disjoint cycles $c_1=(i_0,\ldots,i_{d-1})$
and $c_2=(j_0,\ldots,j_{d-1})$ of $|\pi|$
such that $|w(i_t)|=j_t$ (and vice versa) for every $0\le t\le d-1$.
In this case
\begin{eqnarray*} 
\pi w = w \bar{\pi} &\Longleftrightarrow&
z_w(i_t) + z_\pi(j_t) = -z_\pi(i_t) + z_w(i_{t+1}) \qquad\text{ and }\\ & &
z_w(j_t) + z_\pi(i_t) = -z_\pi(j_t) + z_w(j_{t+1}) \qquad(0\le t\le d-1),
\end{eqnarray*} 
where $t+1$ is computed mod $d$.
Summing any of these over $0\le t\le d-1$ gives
$$
z_\pi(c_1) + z_\pi(c_2) = \sum\limits_{t=0}^{d-1} [z_\pi(i_t) + z_\pi(j_t)] =0.
$$
Given $\pi$, a choice of $z_w(i_0)\in \ZZ_r$ determines uniquely
$z_w(i_t)$ and $z_w(j_t)$ for all $0\le t\le d-1$
(since $z_w(j_t)=z_w(i_t)$ by the condition $w=w^t$).

\medskip

\noindent{\bf Case (3):}
$i_0$ and $|w(i_0)|$ are distinct but belong to the same cycle of $|\pi|$.

\noindent
This is possible only if $d=2e$ is even.
Then there is a cycle $c=(i_0,\dots,i_{2e-1})$ of $|\pi|$ such that
$|w(i_t)|=i_{t+e}$ for every $0\le t\le e-1$.
In this case
$$ 
\pi w = w \bar{\pi} \iff 
z_w(i_t) + z_\pi(i_{t+e}) = -z_\pi(i_t) + z_w(i_{t+1}) \qquad(0\le t\le 2e-1),
$$ 
where $t+e$ and $t+1$ are computed mod $d$.
Summing over $0\le t\le e-1$, remembering that $z_w(i_{t+e})=z_w(i_t)$
since $w=w^t$, yields
$$
z_\pi(c) = \sum\limits_{t=0}^{2e-1} z_\pi(i_t) = 0,
$$
and a choice of $z_w(i_0)\in \ZZ_r$ determines uniquely
$z_w(i_t)$ for all $0\le t\le 2e-1$.

\medskip

Summing up, each $w\in \St(\pi)$ defines a partition $P$ of the set of cycles of $\pi$
into pairs (Cases (2) and (3)) and singletons (Case (1)), with the same color restrictions
as in the definition of $\Pi^{2,1}(\bc)$; but, as we shall see, equality~(\ref{e.eq52})
holds on the level of a single partition $P$ only if both $r$ and $d$ are odd.
Further considerations will be needed in the other cases.

Approaching the actual computation of signs, let us make one further
simplification. Since both sides of~(\ref{e.eq52}) are class functions of $\pi$
(the left-hand-side being the trace of $\rho(\pi)$),
we may assume that
$$
|\pi| = (1,2,\ldots,d)(d+1,\ldots,2d) \cdots ((m-1)d+1,\ldots,md).
$$
Note that $\sign_o(\pi,w)$ depends only on $|\pi|$ and $|w|$.

\begin{observation}\label{t.sign_o}
Consider the above 3 cases.
\begin{itemize}
\item[(1)]
If
$|\pi|=(1,2,\ldots,d)$
and
$|w|=(1)(2)\cdots (d)$
then
$$
\sign_o(\pi,w)=1.
$$
\item[(2)]
If
$|\pi|=(1,2,\ldots,d)(d+1,\ldots,2d)$, $0\le i\le d-1$,
and
$|w|=(1,d+1)(2,d+2) \cdots (d,2d)$ (for $i=0$) or
$|w|=(1,d+1+i)(2,d+2+i) \cdots (d-i,2d)(d-i+1,d+1) \cdots (d,d+i)$
(for $i>0$) then
$$
\sign_o(\pi,w)=1.
$$
\item[(3)]
If $d=2e$ is even,
$|\pi|= (1,2,\ldots,2e)$
and
$|w|=(1,e+1)(2,e+2)\cdots (e,2e)$
then
$$
\sign_o(\pi,w)=-1.
$$
\end{itemize}
\end{observation}

If $d$ is odd then Case (3) does not occur. It then follows from
Observation~\ref{t.sign_o}(1)(2) that $\sign_o(\pi,w) = 1$ for all $w\in \St(\pi)$.
Thus 
$$
\sum_{w\in \St(\pi)} \sign_o(\pi,w)=
\# \St(\pi)=
\sum_{P\in\Pi^{2,1}(\bc)} (dr)^{n_2(P)} r^{n_1(P)},
$$
where the second equality follows from enumeration of all possible $w$ according to
Cases (1) and (2) above.

\medskip

Assume now that $d=2e$ is even. In this case $sign_o(\pi,w)$ may be negative, and
cancellations will occur.
Define a function $\varphi: \St(\pi) \to \St(\pi)$ as follows:
let $w\in \St(\pi)$. Each of the $m$ cycles of $|\pi|$ belongs, with respect to $|w|$,
to one of the Cases (1), (2) and (3).
If all the cycles belong to Case (2), define $\varphi(w) := w$.
Otherwise, let $c_i = ((i-1)d+1,\ldots,(i-1)d+d)$ be the first cycle of $|\pi|$ that
belongs to Cases (1) or (3). If it belongs to Case (1), i.e., if $|w(j)| = j$ for
all $(i-1)d+1\le j\le (i-1)d+d$, define $w' = \varphi(w)$ by
$$
z_{w'}(j) := z_w(j)\qquad(\forall j)
$$
and
$$
|w'(j)| :=
\begin{cases}
j+e,& \text{ if } (i-1)d+1\le j\le (i-1)d+e;\\
j-e,& \text{ if } (i-1)d+e+1\le j\le (i-1)d+2e;\\
|w(j)|,& \text{ otherwise}.
\end{cases}
$$
If the first cycle belongs to Case (3), i.e., if $|w(j)| = j\pm e$ for
all $(i-1)d+1\le j\le (i-1)d+d$, define $w' = \varphi(w)$ by
$$
z_{w'}(j) := z_w(j)\qquad(\forall j)
$$
and
$$
|w'(j)| :=
\begin{cases}
j,& \text{ if } (i-1)d+1\le j\le (i-1)d+d;\\
|w(j)|,& \text{ otherwise}.
\end{cases}
$$
Thus $\varphi$ toggles one of the cycles of $\pi$ between Cases (1) and (3).
It is easy to see that $\varphi$ is a ``sign-reversing involution'' on
$\St(\pi)$, i.e., an involution satisfying
$$
\sign_o(\pi,\varphi(w)) =
\begin{cases}
\sign_o(\pi,w), & \text{ if } \varphi(w) = w;\\
-\sign_o(\pi,w), & \text{ otherwise}
\end{cases}
\qquad(\forall w\in\St(\pi)).
$$
Thus the signs of elements $w\in \St(\pi)$ with $\varphi(w)\ne w$ cancel each other,
whereas $\sign_o(\pi,w) = 1$ when $\varphi(w) = w$.
It follows that, for even $d$,
$$
\sum_{w\in \St(\pi)} \sign_o(\pi,w)=
\#\{w\in \St(\pi)\,|\,\varphi(w) = w\}=
\sum_{P\in\Pi^{2}(\bc)} (dr)^{n_2(P)}.
$$
This completes the proof of Theorem~\ref{t.model} for odd $r$.

\subsection{The Case of Even $r$}

The proof is in general similar to the proof for odd $r$, described above.
We shall focus on the differences.

\subsubsection{Part 1}
Again, we first prove that $\rho$ extends to a group homomorphism.
Partition the set $\ZZ_r = [0,r-1]$ into two complementary ``arcs''
(or ``intervals'') $[0,r/2-1]$ and $[r/2,r-1]$.

\begin{defn}\label{d.sign_e}
For an element $\pi\in G(r,n)$ and an absolute involution $w \in I_{r,n}$ let
\begin{eqnarray*}
B(\pi,w) := \{i:\ & &|w(i)|=i,\ z_w(i) = 2k_w(i)+1 \text{ is odd with }\\
& & k_w(i)\in[0,r/2-1] \text{ and } k_w(i)+z_\pi(i)\in [r/2,r-1]\},
\end{eqnarray*}
and define
$$
\sign_e(\pi,w):= (-1)^{\# B(\pi,w)} \cdot (-1)^{\#(\Inv(|\pi|) \cap \Pair(|w|))}.
$$
The second factor is the same as in Definition~\ref{d.sign_o}.
\end{defn}

Define a map $\rho: G(r,n)\to GL(V_{r,n})$ by
$$ 
\rho(\pi)C_w := \sign_e(\pi,w) \cdot C_{\pi w\pi^t}
\qquad (\forall \pi\in G(r,n), w\in I_{r,n}).
$$ 

\medskip

Again, one can verify that this definition of $\rho$ coincides,
on the set $S$ of generators of $G(r,n)$,
with the previous definition~(\ref{e.linear_action}).
It thus suffices to show that $\rho$ is a group homomorphism,
namely that
$$ 
\sign_e(\pi_2\pi_1,w)=\sign_e(\pi_1,w)\cdot \sign_e(\pi_2,\pi_1 w{\pi_1^t}).
$$ 
By the proof of Part 1 for odd $r$, it
suffices to prove that
$$
(-1)^{\# B(\pi_2\pi_1,w)} = (-1)^{\#B(\pi_1,w)} \cdot (-1)^{\#B(\pi_2, \pi_1 w \pi_1^t)}.
$$
Indeed, letting again $X[\text{condition}]$ be $-1$ if the condition holds and $1$ otherwise,
it suffices to prove that for every $1\le i\le n$
\begin{equation}\label{e.X2}
X[i\in B(\pi_2\pi_1,w)] = X[i\in B(\pi_1,w)] \cdot X[|\pi_1(i)|\in B(\pi_2,\pi_1 w\pi_1^t)].
\end{equation}
Note that
$$
|w(i)|=i \iff |\pi_1 w\pi_1^t(\pi_1(i))|=|\pi_1(i)|
$$
and
$$
z_w(i) \text{ is odd } \iff z_{\pi_1 w\pi_1^t}(|\pi_1(i)|) \text{ is odd.}
$$
Hence, in order to prove (\ref{e.X2}), we can assume that $|w(i)|=i$ and that $z_w(i)$ is odd.
Denote $z_w(i)=2k+1$ with $k=k_w(i)\in[0,r/2-1]\subseteq \ZZ_r$, $m_1 := z_{\pi_1}(i)\in \ZZ_r$
and $m_2 := z_{\pi_2}(|\pi_1(i)|)\in\ZZ_r$.
Note that $z_{\pi_1 w \pi_1^t}(|\pi_1(i)|) =2(k+m_1)+1$.
Define
$$
k_1 :=
\begin{cases}
k+m_1, & \text{ if } k+m_1\in[0,r/2-1];\\
k+m_1-r/2, & \text{ otherwise},
\end{cases}
$$
where all operations are in $\ZZ_r$.
Equation~(\ref{e.X2}) now reduces to
$$
X[k+m_1+m_2\in [r/2,r-1]] = X[k+m_1\in [r/2,r-1]] \cdot X[k_1+m_2\in [r/2,r-1]].
$$
This is obviously true, completing the first part of the proof.

\subsubsection{Part 2.}
By the arguments used for odd $r$, in order to show that $\rho$ is a model for $G(r,n)$
it suffices to show that, for every $\pi\in G(r,md)$ with $|\pi|$ of cycle type $d^m$,
\begin{equation}\label{e.eq53}
\sum\limits_{w\in \St(\pi)} \sign_e(\pi,w) =
\sum_{P\in\Pi^{2(,1)}(w)} (dr)^{n_2(P)} r^{n_1(P)}.
\end{equation}
Again, for $w\in\St(\pi)$ and $i_0\in[md]$,
there are 3 cases to consider.

\medskip

\noindent{\bf Case (1):} $|w(i_0)|=i_0$.

\noindent
There exists a cycle $c=(i_0,\dots,i_{d-1})$ of $|\pi|$
such that $|w(i_t)|=i_t$ for all $0\le t\le d-1$, but the equation
$$
2z_\pi(c) = 2\sum\limits_{t=0}^{d-1} z_\pi(i_t)=0
$$
(in $\ZZ_r$) implies, for even $r$, only that
$$
z_\pi(c)\in \{0, r/2\}.
$$
The case $z_\pi(c) = r/2$ does not fit the color restrictions for $P\in\Pi^{2(,1)}(\bc)$,
if the cycle $c$ is taken as a singleton.

\medskip

\noindent{\bf Cases (2), (3):} $|w(i_0)|\ne i_0$.

\noindent
The analysis here is exactly the same as for odd $r$.

\medskip

Again, since both sides of~(\ref{e.eq53}) are class functions, we can choose $\pi$
to be any representative of its conjugacy class in $G(r,n)$. We shall require that
$$
|\pi| = (1,2,\ldots,d)(d+1,\ldots,2d) \cdots ((m-1)d+1,\ldots,md)
$$
and, moreover, that $z_\pi(i) = 0$ unless $i\equiv 1 \pmod d$; i.e., that $z_\pi$
of each cycle is concentrated in its smallest element.

The analogues of Observation~\ref{t.sign_o}(2)(3) are exactly as for odd $r$,
since $|w|$ has no fixed points in these cases, and therefore $\#B(\pi,w) = 0$.

Consider now the situation in Observation~\ref{t.sign_o}(1).
By the above assumptions on $\pi$,
the only possible member of $B(\pi,w)$ is $i=1$ (since $z_\pi(i)=0$ for $i\ne 1$).
Also, $z_\pi(1)$ is either $0$ or $r/2$.
If $z_\pi(1) = 0$ then $\#B(\pi,w) = 0$ for any $w\in\St(\pi)$.
If $z_\pi(1) = r/2$ then $\#B(\pi,w) = 1$ for any $w\in\St(\pi)$ with $z_w(1)$ odd,
but $\#B(\pi,w) = 0$ for any $w\in\St(\pi)$ with $z_w(1)$ even.
Since $z_w(1)$ uniquely determines $w$, we conclude that for
$|\pi|$ with a single cycle, $|w|$ of Case (1),
$$
\sum_{w\in \St(\pi)} \sign_e(\pi,w) =
\begin{cases}
r, & \text{ if } z_\pi(1) = 0;\\
0, & \text{ if } z_\pi(1) = r/2.\\
\end{cases}
$$
Thus, for $|\pi|$ of cycle type $d^m$, if $z_\pi(c) = r/2$ for some cycle $c$
then the total sign contribution of all $w\in\St(\pi)$ which are of Case (1)
on $c$ is zero, and this set of $w$-s can be discarded.
The rest of the proof is exactly as for odd $r$.

\qed

\section{Remarks and Questions}\label{s.final-section}

%

\subsection{RSK for wreath products}




If two absolute involutions are mapped to the same shape then their
corresponding uncolored involutions in $S_n$ are conjugate. It
results that all absolute involutions with fixed number of cycles
form an invariant submodule of $V_{r,n}$.

We conjecture that the RSK on absolute involutions is compatible
with the decomposition of the Gelfand model, namely

\begin{conjecture}
The submodule spanned by all absolute involutions with fixed
number of cycles is a multiplicity free sum of all irreducible
$\ZZ_r\wr S_n$ representations whose shape is obtained from these
absolute involutions via the colored RSK.
\end{conjecture}

\subsection{Construction of Models}

%
%
%

It should be noted that the construction for type $B$ may be adapted to $D_{2n+1}$.

\begin{question}
Give a construction of a Gelfand model for $D_{2n}$; for general
complex reflection groups; for affine Weyl groups.
\end{question}

Finally, a $q$-deformation of the model for $S_n=\ZZ_1\wr  S_n$,
which gives a model for the Iwahori Hecke algebra of type $A$, was
described in~\cite{APR}. A construction of  a model for the
Iwahori Hecke algebra of $\ZZ_r\wr S_n$ is desired.

\section{Acknowledgements}
The authors thank Eitan Sayag for helpful remarks and references.


\end{document}